\documentclass[10pt]{article}

\usepackage{amsmath,amsbsy,amssymb,amsthm}
\usepackage{a4wide}
\usepackage{graphicx,psfrag}

\theoremstyle{plain}
\newtheorem{theorem}{Theorem}[section]

\theoremstyle{remark}
\newtheorem{remark}{Remark}

\theoremstyle{definition}
\newtheorem{definition}{Definition}

\def\a{{\alpha(\cdot,\cdot)}}
\def\b{{\beta(\cdot,\cdot)}}
\def\t{\tau}
\def\e{\epsilon}
\def\LI{{_aI_t^{\a}}}
\def\RI{{_tI_b^{\a}}}
\def\LDa{{_aD_t^{\a}}}
\def\LDb{{_aD_t^{\b}}}
\def\RDa{{_tD_b^{\a}}}
\def\RDb{{_tD_b^{\b}}}
\def\LC{{^C_aD_t^{\a}}}
\def\RCa{{^C_tD_b^{\a}}}
\def\RCb{{^C_tD_b^{\b}}}
\def\DC{{^CD_\gamma^{\a,\b}}}

\begin{document}

\title{Constrained fractional variational problems of variable order\thanks{This 
is a preprint of a paper whose final and definite form will appear 
in the \emph{IEEE/CAA Journal of Automatica Sinica}, ISSN 2329-9266. 
Submitted 05-Sept-2015; Revised 19-April-2016; Accepted 22-June-2016.}}

\author{Dina Tavares$^{{\rm a}, {\rm b}}$\\
{\tt dtavares@ipleiria.pt}
\and Ricardo Almeida$^{\rm b}$\\
{\tt ricardo.almeida@ua.pt}
\and Delfim F. M. Torres$^{\rm b}$\\
{\tt delfim@ua.pt}}

\date{$^{a}$ESECS, Polytechnic Institute of Leiria, 2410--272 Leiria, Portugal\\[0.3cm]
$^{b}$\text{Center for Research and Development in Mathematics and Applications (CIDMA)}\\
Department of Mathematics, University of Aveiro, 3810--193 Aveiro, Portugal}

\maketitle


\begin{abstract}
Isoperimetric problems consist in minimizing or maximizing a cost functional
subject to an integral constraint. In this work, we present two fractional
isoperimetric problems where the Lagrangian depends on a combined Caputo
derivative of variable fractional order and we present a new variational
problem subject to a holonomic constraint. We establish necessary optimality
conditions in order to determine the minimizers of the fractional problems.
The terminal point in the cost integral, as well the terminal state,
are considered to be free, and we obtain corresponding natural boundary conditions.

\bigskip


\smallskip

\end{abstract}


\section{Introduction}

Many real world phenomena are better described by noninteger order derivatives.
In fact, fractional derivatives have unique characteristics that may model
certain dynamics more efficiently. To start, we can consider any real order
for the derivatives, and thus we are not restricted to integer-order derivatives
only. Secondly, they are nonlocal operators, in opposite to the usual derivatives,
containing memory. With the memory property one can take into account the past
of the processes. This subject, called \emph{Fractional Calculus}, although
as old as ordinary calculus itself, only recently has found numerous applications
in mathematics, physics, mechanics, biology and engineering. The order of the
derivative is assumed to be fixed along the process, that is, when determining
what is the order $\alpha>0$ such that the solution of the fractional differential
equation $D^\alpha y(t)=f(t,y(t))$ better approaches the experimental data,
we consider the order to be a fixed constant. Of course, this may not be the
best option, since trajectories are a dynamic process, and the order may vary.
So, the natural solution to this problem is to consider the order to be a function,
$\alpha(\cdot)$, depending on time. Then we may seek what is the best function
$\alpha(\cdot)$ such that the variable order fractional differential equation
$D^{\alpha(\cdot)} y(t)=f(t,y(t))$ better describes the model. This approach
is very recent, and many work has to be done for a complete study of the subject
(see, e.g., \cite{Atangana,Coimbra,SamkoRoss,Shenga,Valerio}).

The most common fractional operators considered in the literature
take into account the past of the process: they are usually called
left fractional operators. But in some cases we may be also interested
in the future of the process, and the computation of $\alpha(\cdot)$
to be influenced by it. In that case, right fractional derivatives
are then considered. Our goal is to develop a theory where both fractional
operators are taken into account, and for that we define a combined fractional
variable order derivative operator that is a linear combination of the left
and right fractional derivatives. For studies with fixed fractional order
see \cite{Malin:Tor,MyID:206,MyID:207}.

Variational problems are often subject to one or more constraints.
For example, isoperimetric problems are optimization problems where
the admissible functions are subject to integral constraints. This
direction of research has been recently investigated in \cite{Tavares},
where variational problems with dependence on a combined Caputo derivative
of variable fractional order are considered and necessary optimality conditions
deduced. Here variational problems are considered subject to integral
or holomonic constraints.

The text is organized in four sections. In Section~\ref{sec:FC} we review
some important definitions and results about combined Caputo derivative
of variable fractional order, and present some properties that will be need
in the sequel. For more on the subject we refer the interested reader to
\cite{Tatiana:IDOTA2011,Od,SamkoRoss}. In Section~\ref{sec:Iso} we present
two different isoperimetric problems and we study necessary optimality conditions
in order to determine the minimizers for each of the problems. We end
Section~\ref{sec:Iso} with an example. In Section~\ref{sec:Hol}
we consider a new variational problem subject to a holonomic constraint.


\section{Fractional calculus of variable order}
\label{sec:FC}

In this section we collect definitions and preliminary results on fractional
calculus, with variable fractional order, needed in the sequel. The variational
fractional order is a continuous function of two variables, $\alpha:[a,b]^2\to(0,1)$.
Let $x:[a,b]\to\mathbb{R}$. Two different types of fractional derivatives
are considered.

\begin{definition}[Riemann--Liouville fractional derivatives]
The left and right Riemann--Liouville fractional derivatives
of order $\a$ are defined respectively by
$$
\LDa x(t)=\frac{d}{dt}\int_a^t
\frac{1}{\Gamma(1-\alpha(t,\t))}(t-\t)^{-\alpha(t,\t)}x(\t)d\t
$$
and
$$
\RDa x(t)=\frac{d}{dt}\int_t^b
\frac{-1}{\Gamma(1-\alpha(\t,t))}(\t-t)^{-\alpha(\t,t)} x(\t)d\t.
$$
\end{definition}

\begin{definition}[Caputo fractional derivative]
The left and right Caputo fractional derivatives of order $\a$ are defined
respectively by
$$
\LC x(t)=\int_a^t
\frac{1}{\Gamma(1-\alpha(t,\t))}(t-\t)^{-\alpha(t,\t)}x^{(1)}(\t)d\t
$$
and
$$
\RCa x(t)=\int_t^b
\frac{-1}{\Gamma(1-\alpha(\t,t))}(\t-t)^{-\alpha(\t,t)}x^{(1)}(\t)d\t.
$$
\end{definition}

Of course the fractional derivatives just defined are linear operators.
The next step is to define a new fractional derivative,
combining the previous ones into a single one.

\begin{definition}
\label{def1}
Let $\alpha, \, \beta: [a,b]^2\rightarrow(0,1)$ be the fractional orders,
and define the constant vector $\gamma=(\gamma_1,\gamma_2) \in [0,1]^{2}$.
The combined Riemann--Liouville fractional derivative of a function $x$
is defined by
$$
D_\gamma^{\a,\b}x(t)=\gamma_1 \, \LDa x(t)+\gamma_2 \, \RDb x(t).
$$
The combined Caputo fractional derivative of a functions $x$ is defined by
$$
^{C}D_\gamma^{\a,\b}x(t)=\gamma_1 \, \LC x(t)+\gamma_2 \, \RCb x(t).
$$
\end{definition}

For the sequel, we also need the generalization of fractional integrals
for a variable order.

\begin{definition}[Riemann--Liouville fractional integrals]
The left and right Riemann--Liouville fractional integrals
of order $\a$ are defined respectively by
$$
\LI x(t)=\int_a^t \frac{1}{\Gamma(\alpha(t,\t))}(t-\t)^{\alpha(t,\t)-1}x(\t)d\t
$$
and
$$
\RI x(t)=\int_t^b\frac{1}{\Gamma(\alpha(\t,t))}(\t-t)^{\alpha(\t,t)-1} x(\t)d\t.
$$
\end{definition}

We remark that in contrast to the fixed fractional order case,
variable-order fractional integrals are not the inverse operation
of the variable-order fractional derivatives.

For the next section,
we need the following fractional integration by parts formulas.

\begin{theorem}[Theorem~3.2 of \cite{Od}]
\label{thm:FIP}
If $x,y \in C^1[a,b]$, then
$$
\int_{a}^{b}y(t) \, \LC x(t)dt
=\int_a^b x(t) \, {\RDa}y(t)dt
+\left[x(t) \, {_tI_b^{1-\a}}y(t) \right]_{t=a}^{t=b}
$$
and
$$
\int_{a}^{b}y(t) \, {\RCa}x(t)dt=\int_a^b x(t) \, {\LDa} y(t)dt
-\left[x(t) \, {_aI_t^{1-\a}}y(t)\right]_{t=a}^{t=b}.
$$
\end{theorem}


\section{Fractional isoperimetric problems}
\label{sec:Iso}

Consider the set
$$
D=\left\{ (x,t)\in  C^1([a,b])\times [a,b] : \DC x(t) \,
\mbox{exists and is continuous on }\, [a,b]\right\},
$$
endowed with the norm
$$
\|(x,t)\|:=\max_{a\leq t \leq  b}|x(t)|
+\max_{a\leq t \leq b}\left| \DC x(t)\right|+|t|.
$$
Throughout the text, we denote by $\partial_i z$ the partial derivative
of a function $z:\mathbb{R}^{3} \rightarrow\mathbb{R}$ with respect
to its $i$th argument. Also, for simplification, we consider the operator
$$
[x]_\gamma^{\alpha, \beta}(t):=\left(t, x(t), \DC x(t)\right).
$$

The main problem of the fractional calculus of variations with variable order
is described as follows. Let $L:C^{1}\left([a,b]\times \mathbb{R}^2 \right)
\to\mathbb{R}$ and consider the functional
$\mathcal{J}:D\rightarrow \mathbb{R}$ of the form
\begin{equation}
\label{funct1}
\mathcal{J}(x,T)=\int_a^T L[x]_\gamma^{\alpha, \beta}(t) dt + \phi(T,x(T)),
\end{equation}
where $\phi:[a,b]\times \mathbb{R}\to\mathbb{R}$ is of class $C^1$. In the sequel,
we need the auxiliary notation of the dual fractional derivative:
\begin{equation}
\label{aux:FD}
D_{\overline{\gamma},c}^{\b,\a}=\gamma_2 \, {_aD_t^{\b}}+\gamma_1 \, {_tD_c^{\a}},
\quad \mbox{where} \quad \overline{\gamma}=(\gamma_2,\gamma_1)
\quad \mbox{and} \quad c\in(a,b].
\end{equation}

\begin{remark}
Fractional derivatives \eqref{aux:FD}
can be regarded as a generalization of usual
fractional derivatives. For advantages of applying 
them to fractional variational problems see
\cite{MR2314503,MR2944107,MyID:207}.
\end{remark}

In \cite{Tavares} we obtained necessary conditions that every local
minimizer of functional $\mathcal{J}$ must fulfill.

\begin{theorem}[See \cite{Tavares}]
\label{thm:tav:paper01}
If $(x,T)\in D$ is a local minimizer of functional \eqref{funct1},
then $(x,T)$ satisfies the fractional differential equation
$$
\partial_2 L[x]_\gamma^{\alpha, \beta}(t)
+D{_{\overline{\gamma},T}^{\b,\a}}\partial_3 L[x]_\gamma^{\alpha, \beta}(t)=0
$$
on $[a,T]$ and
$$
\gamma_2\left({\LDb}\partial_3 L[x]_\gamma^{\alpha, \beta}(t)
-{ _TD{_t^{\b}}\partial_3 L[x]_\gamma^{\alpha, \beta}(t)}\right)=0
$$
on $[T,b]$.
\end{theorem}

\begin{remark}
In general, analytical solutions to fractional
variational problems are hard to find.
For this reason, numerical methods are often used.
The reader interested  in this subject is referred 
to \cite{MR3443073,MR3434464} and references therein.
\end{remark}

Fractional differential equations as the ones given
by Theorem~\ref{thm:tav:paper01}, are known in the literature
as fractional Euler--Lagrange equations, and they provide us with a method
to determine the candidates for solutions of the problem addressed. Solutions
of such fractional differential equations are called extremals for the functional.
In this paper, we proceed the study initiated in \cite{Tavares}
by considering additional constraints to the problems.
We will deal with two types of isoperimetric problems, which we now describe.


\subsection{Problem I}

The fractional isoperimetric problem of the calculus of variations consists to
determine the local minimizers of $\mathcal{J}$ over all $(x,T)\in D$
satisfying a boundary condition
\begin{equation}
\label{bcxa}
x(a)=x_a
\end{equation}
for a fixed $x_a\in \mathbb{R}$ and an integral constraint of the form
\begin{equation}
\label{IsoConst_1}
\int_a^T g[x]_\gamma^{\alpha, \beta}(t)dt=\psi(T),
\end{equation}
where $g:C^{1}\left([a,b]\times \mathbb{R}^2 \right)\to\mathbb{R}$
and $\psi: [a,b]\to \mathbb{R}$ are two differentiable functions. The terminal
time $T$ and terminal state $x(T)$ are free. In this problem, the condition
of the form \eqref{IsoConst_1} is called an isoperimetric constraint.
The next theorem gives fractional necessary optimality
conditions to this isoperimetric problem.

\begin{theorem}
\label{teo1}
Suppose that $(x,T)$ gives a local minimum  for functional \eqref{funct1} on
$D$ subject to the boundary condition \eqref{bcxa} and the isoperimetric
constraint \eqref{IsoConst_1}. If $(x,T)$ does not satisfies the Euler--Lagrange
equations with respect to the isoperimetric constraint, that is,
if one of the two following conditions are not verified,
\begin{equation}
\label{eq:eq1}
\partial_2 g[x]_\gamma^{\alpha, \beta}(t) + D_{\overline{\gamma},T}^{\b,\a}
\partial_3g[x]_\gamma^{\alpha, \beta}(t)=0, \quad t\in[a,T],
\end{equation}
or
\begin{equation}
\label{eq:eq2}
\gamma_2\left[ _aD_t^{\b} \partial_3 g[x]_\gamma^{\alpha, \beta}(t)
- {_TD_t}^{\b} \partial_3 g[x]_\gamma^{\alpha, \beta}(t) \right]=0,
\quad t\in[T,b],
\end{equation}
then there exists a constant $\lambda$ such that, if we define the function
$F:[a,b]\times \mathbb{R}^2\to\mathbb{R}$ by $F= L-\lambda g$,
$(x,T)$ satisfies the fractional Euler--Lagrange equations
\begin{equation}
\label{ELeq_1}
\partial_2 F[x]_\gamma^{\alpha, \beta}(t)
+D{_{\overline{\gamma},T}^{\b,\a}}\partial_3 F[x]_\gamma^{\alpha, \beta}(t)=0
\end{equation}
on the interval $[a,T]$ and
\begin{equation}
\label{ELeq_2}
\gamma_2\left({\LDb}\partial_3F[x]_\gamma^{\alpha, \beta}(t)
-{ _TD{_t^{\b}}\partial_3F[x]_\gamma^{\alpha, \beta}(t)}\right)=0
\end{equation}
on the interval $[T,b]$. Moreover,
$(x,T)$ satisfies the transversality conditions
\begin{equation}
\label{CT1}
\begin{cases}
F[x]_\gamma^{\alpha, \beta}(T)+\partial_1\phi(T,x(T))
+\partial_2\phi(T,x(T))x'(T)+\lambda \psi'(T)=0,\\
\left[\gamma_1 \, {_tI_T^{1-\a}\partial_3F[x]_\gamma^{\alpha, \beta}(t)}
-{\gamma_2 \, {_TI_t^{1-\b}\partial_3F[x]_\gamma^{\alpha, \beta}(t)}}\right]_{t=T}
+\partial_2 \phi(T,x(T))=0,\\
\gamma_2 \left[ {_TI_t^{1-\b}}\partial_3 F[x]_\gamma^{\alpha, \beta}(t)
-{_aI_t^{1-\b}\partial_3F[x]_\gamma^{\alpha, \beta}(t)}\right]_{t=b}=0.
\end{cases}
\end{equation}
\end{theorem}

\begin{proof}
Consider variations of the optimal solution $(x,T)$ of the type
\begin{equation}
\label{av}
(x^{*},T^{*})=\left(x+\e_{1}{h_{1}}+\e_{2}{h_{2}},T+\e_{1}\Delta{T}\right),
\end{equation}
where, for each $i \in \{1,2\}$, $\e_{i} \in\mathbb{R}$ is a small parameter,
$h_{i}\in C^1([a,b])$ satisfies $h_{i}(a)=0$, and $\triangle T \in\mathbb{R}$.
The additional term $\e_{2}{h_{2}}$ must be selected so that the admissible
variations $(x^{*},T^{*})$ satisfy the isoperimetric constraint \eqref{IsoConst_1}.
For a fixed choice of $h_{i}$, let
$$
i(\e_{1},\e_{2})=\int_a^{T+\e_{1} \triangle T}
g[x^{*}]_\gamma^{\alpha, \beta}(t)dt-\psi(T+\e_1\triangle T).
$$
For $\e_{1}=\e_{2}=0$, we obtain that
$$
i(0,0)=\int_a^{T} g[x]_\gamma^{\alpha, \beta}(t)dt-\psi(T)=\psi(T)-\psi(T)=0.
$$
The derivative $\dfrac{\partial i}{\partial \e_{2}}$ is given by
\begin{equation*}
\dfrac{\partial i}{\partial \e_{2}}=\int_a^{T+\e_{1} \triangle T} \left(
\partial_2 g[x^{*}]_\gamma^{\alpha, \beta}(t) h_{2}(t)
+ \partial_3 g[x^{*}]_\gamma^{\alpha, \beta}(t) \DC h_{2}(t) \right)dt.
\end{equation*}
For $\e_{1}=\e_{2}=0$ one has
\begin{equation}
\label{funct5}
\left. \dfrac{\partial i}{\partial \e_{2}} \right|_{(0,0)}=\int_a^{T} \left(
\partial_2 g[x]_\gamma^{\alpha, \beta}(t) h_{2}(t) \
+ \partial_3 g[x]_\gamma^{\alpha, \beta}(t) \DC h_{2}(t) \right)dt.
\end{equation}
The second term in \eqref{funct5} can be written as
\begin{equation}
\label{term2}
\begin{split}
\int_a^T  &\partial_3 g[x]_\gamma^{\alpha, \beta}(t) \DC h_{2}(t) dt\\
&=\int_a^T \partial_3 g[x]_\gamma^{\alpha, \beta}(t)\left[\gamma_1
\, \LC h_{2}(t)+\gamma_2 \, \RCb h_{2}(t)\right]dt\\
&=\gamma_1 \int_a^T \partial_3 g[x]_\gamma^{\alpha, \beta}(t)\LC h_{2}(t)dt  \\
&\quad + \gamma_2 \left[ \int_a^b \partial_3
g[x]_\gamma^{\alpha, \beta}(t) \RCb h_{2}(t)dt
- \int_T^b \partial_3 g[x]_\gamma^{\alpha, \beta}(t) \RCb h_{2}(t)dt \right].
\end{split}
\end{equation}
Using the fractional integrating by parts formula, \eqref{term2} is equal to
\begin{equation*}
\begin{split}
\int_a^T & h_{2}(t) \left[ \gamma_1 {_tD_T}^{\a} \partial_3
g[x]_\gamma^{\alpha, \beta}(t) + \gamma_2 {_aD_t}^{\b}
\partial_3 g[x]_\gamma^{\alpha, \beta}(t) \right] dt\\
+ &  \int_T^b \gamma_2 h_{2}(t) \left[ _aD_t^{\b} \partial_3
g[x]_\gamma^{\alpha, \beta}(t) - {_TD_t}^{\b} \partial_3
g[x]_\gamma^{\alpha, \beta}(t) \right] dt\\
+ & \Biggl[ h_{2}(t) \left( \gamma_1 {{_tI_T}}^{1-\a} \partial_3
g[x]_\gamma^{\alpha, \beta}(t) - \gamma_2 {_TI_t}^{1-\b} \partial_3
g[x]_\gamma^{\alpha, \beta}(t) \Biggr) \right]_{t=T}\\
+ & \Biggl[ \gamma_2 h_{2}(t)  \left( {{_TI_t}}^{1-\b} \partial_3
g[x]_\gamma^{\alpha, \beta}(t) - {_aI_t}^{1-\b} \partial_3
g[x]_\gamma^{\alpha, \beta}(t) \Biggr) \right]_{t=b}.
\end{split}
\end{equation*}
Substituting these relations into \eqref{funct5}, and considering the fractional
operator $D_{\overline{\gamma},c}^{\b,\a}$
as defined in \eqref{aux:FD}, we obtain that
\begin{equation*}
\begin{split}
\left.\dfrac{\partial i}{\partial \e_{2}} \right|_{(0,0)}
= \int_a^T & h_{2}(t) \left[ \partial_2 g[x]_\gamma^{\alpha, \beta}(t)
+ D_{\overline{\gamma},T}^{\b,\a}\partial_3
g[x]_\gamma^{\alpha, \beta}(t) \right] dt\\
&+ \int_T^b \gamma_2 h_{2}(t) \left[ _aD_t^{\b} \partial_3
g[x]_\gamma^{\alpha, \beta}(t) - {_TD_t}^{\b} \partial_3
g[x]_\gamma^{\alpha, \beta}(t) \right] dt\\
&+ \Biggl[ h_{2}(t) \left( \gamma_1 {{_tI_T}}^{1-\a} \partial_3
g[x]_\gamma^{\alpha, \beta}(t) - \gamma_2 {_TI_t}^{1-\b} \partial_3
g[x]_\gamma^{\alpha, \beta}(t) \Biggr) \right]_{t=T}\\
&+ \Biggl[ \gamma_2 h_{2}(t)  \left( {{_TI_t}}^{1-\b} \partial_3
g[x]_\gamma^{\alpha, \beta}(t) - {_aI_t}^{1-\b} \partial_3
g[x]_\gamma^{\alpha, \beta}(t) \Biggr) \right]_{t=b}.
\end{split}
\end{equation*}
Since \eqref{eq:eq1} or \eqref{eq:eq2} fails, there exists a function $h_{2}$
such that
$$
\left.\dfrac{\partial i}{\partial \e_{2}} \right|_{(0,0)}\neq 0.
$$
In fact, if not, from the arbitrariness of the function $h_2$ and the
fundamental lemma of the calculus of the variations, \eqref{eq:eq1}
and \eqref{eq:eq2} would be verified. Thus, we may apply the implicit
function theorem, that ensures the existence of a function $\e_{2}(\cdot)$,
defined in a neighborhood of zero, such that $i(\e_1,\e_2(\e_1))=0$.
In conclusion, there exists a subfamily of variations of the form \eqref{av}
that verifies the integral constraint \eqref{IsoConst_1}. We now seek to prove
the main result. For that purpose, consider the auxiliary function
$j(\e_1,\e_2)=\mathcal{J}(x^{*},T^{*})$.
By hypothesis, function $j$ attains a local minimum at $(0,0)$ when subject
to the constraint $i(\cdot,\cdot)=0$, and we proved before that
$\nabla i(0,0)\not=0$. Applying the Lagrange multiplier rule, we ensure
the existence of a number $\lambda$ such that
$$
\nabla \left(j(0,0)-\lambda i(0,0)\right)=0.
$$
In particular,
\begin{equation}
\label{funct7}
\dfrac{\partial \left(j-\lambda i\right)}{\partial \e_{1}} (0,0)=0.
\end{equation}
Let $F=L-\lambda g$. The relation \eqref{funct7} can be written as
\begin{equation}
\label{funct8}
\begin{split}
0= &\int_a^Th_{1}(t)\left[\partial_2 F[x]_\gamma^{\alpha, \beta}(t)
+ D_{\overline{\gamma},T}^{\b,\a}\partial_3 F[x]_\gamma^{\alpha, \beta}(t)\right]dt\\
&+ \int_T^b \gamma_2 h_{1}(t) \left[_aD_t^{\b}\partial_3 F[x]_\gamma^{\alpha, \beta}(t)
-{_TD_t^{\b}\partial_3 F[x]_\gamma^{\alpha, \beta}(t)}\right]dt\\
&+ h_{1}(T)\left[\gamma_1 \, {_tI_T^{1-\a}\partial_3F[x]_\gamma^{\alpha, \beta}(t)}
-{\gamma_2 \, {_TI_t^{1-\b}\partial_3F[x]_\gamma^{\alpha, \beta}(t)}}
+\partial_2 \phi(t,x(t))\right]_{t=T}\\
&+\Delta T \left[ F[x]_\gamma^{\alpha, \beta}(t)+\partial_1\phi(t,x(t))
+\partial_2\phi(t,x(t))x'(t)+\lambda\psi'(t)  \right]_{t=T}\\
&+ h_{1}(b) \gamma_2  \left[ _TI_t^{1-\b}\partial_3 F[x]_\gamma^{\alpha, \beta}(t)
-{_aI_t^{1-\b}\partial_3F[x]_\gamma^{\alpha, \beta}(t)} \right]_{t=b}.
\end{split}
\end{equation}
As $h_{1}$ and $\triangle T$ are arbitrary, we can choose $\triangle T=0$ and
$h_{1}(t)=0$ for all $t\in[T,b]$. But $h_{1}$ is arbitrary in $t\in[a,T)$.
Then, we obtain the first necessary condition \eqref{ELeq_1}:
$$
\partial_2 F[x]_\gamma^{\alpha, \beta}(t)
+D{_{\overline{\gamma},T}^{\b,\a}}\partial_3 F[x]_\gamma^{\alpha, \beta}(t)=0
\quad \forall t \in [a,T].
$$
Analogously, considering $\triangle T=0$ and $h_{1}(t)=0$
for all $t\in[a,T]\cup\{b\}$, and $h_{1}$ arbitrary on $(T,b)$,
we obtain the second necessary condition \eqref{ELeq_2}:
$$
\gamma_2\left({\LDb}\partial_3F[x]_\gamma^{\alpha, \beta}(t)
-{ _TD{_t^{\b}}\partial_3F[x]_\gamma^{\alpha, \beta}(t)}\right)=0
\quad \forall t \in [T,b].
$$
As $(x,T)$ is a solution to the necessary conditions \eqref{ELeq_1}
and \eqref{ELeq_2}, then equation \eqref{funct8} takes the form
\begin{equation}
\label{eq_derj3}
\begin{split}
0=&h_{1}(T)\left[\gamma_1 \, {_tI_T^{1-\a}\partial_3F[x]_\gamma^{\alpha, \beta}(t)}
-{\gamma_2 \, {_TI_t^{1-\b}\partial_3F[x]_\gamma^{\alpha, \beta}(t)}}
+\partial_2 \phi(t,x(t))\right]_{t=T}\\
&+\Delta T \left[ F[x]_\gamma^{\alpha, \beta}(t)+\partial_1\phi(t,x(t))
+\partial_2\phi(t,x(t))x'(t)+\lambda\psi'(t)  \right]_{t=T}\\
&+ h_{1}(b)\left[\gamma_2  \left( _TI_t^{1-\b}\partial_3 F[x]_\gamma^{\alpha, \beta}(t)
-{_aI_t^{1-\b}\partial_3F[x]_\gamma^{\alpha, \beta}(t)}\right) \right]_{t=b}.
\end{split}
\end{equation}
Transversality conditions \eqref{CT1} are obtained
for appropriate choices of variations.
\end{proof}

In the next theorem, considering the same Problem I,
we rewrite the transversality conditions \eqref{CT1} in terms of the increment
on time $\Delta T$ and on the increment of space $\Delta x_T$ given by
\begin{equation}
\label{Dxt}
\Delta x_T = (x+h_1)(T+\Delta T)-x(T).
\end{equation}

\begin{theorem}
\label{teo2}
Let $(x,T)$ be a local minimizer to the functional \eqref{funct1} on $D$ subject
to the  boundary condition \eqref{bcxa} and the isoperimetric constraint
\eqref{IsoConst_1}. Then $(x,T)$ satisfies the transversality conditions
\begin{equation}
\label{CT2}
\begin{cases}
F[x]_\gamma^{\alpha, \beta}(T)+\partial_1\phi(T,x(T))+\lambda \psi'(T)\\
\qquad + x'(T) \left[ \gamma_2 {_TI}_t^{1-\b} \partial_3F[x]_\gamma^{\alpha, \beta}(t)
- \gamma_1 {_tI_T^{1-\a} \partial_3F[x]_\gamma^{\alpha, \beta}(t)} \right]_{t=T} =0,\\
\left[\gamma_1 \, {_tI_T^{1-\a}\partial_3F[x]_\gamma^{\alpha, \beta}(t)}
-{\gamma_2 \, {_TI_t^{1-\b}\partial_3F[x]_\gamma^{\alpha, \beta}(t)}}\right]_{t=T}
+\partial_2 \phi(T,x(T))=0,\\
\gamma_2 \left[ {_TI_t^{1-\b}}\partial_3 F[x]_\gamma^{\alpha, \beta}(t)
-{_aI_t^{1-\b}\partial_3F[x]_\gamma^{\alpha, \beta}(t)}\right]_{t=b}=0.
\end{cases}
\end{equation}
\end{theorem}

\begin{proof}
Suppose $(x^{*},T^{*})$ is an admissible variation of the form \eqref{av}
with $\epsilon_1=1$ and $\epsilon_2=0$. Using Taylor's expansion up to first
order for a small $\Delta T$, and restricting the set of variations to those
for which $h_{1}'(T)=0$, we obtain the increment $\Delta x_T$ on $x$:
$$
(x+h_{1})\left( T + \Delta T \right)=(x+h_{1})(T)+x'(T) \Delta T + O(\Delta T)^{2}.
$$
Relation \eqref{Dxt} allows us to express $h_{1}(T)$
in terms of $\Delta T$ and $\Delta x_T$:
$$
h_{1}(T)= \Delta x_T - x'(T) \Delta T + O(\Delta T)^{2}.
$$
Substitution this expression into \eqref{eq_derj3}, and using appropriate
choices of variations, we obtain the new transversality conditions \eqref{CT2}.
\end{proof}

\begin{theorem}
\label{teo3}
Suppose that $(x,T)$ gives a local minimum  for functional \eqref{funct1} on $D$
subject to the  boundary condition \eqref{bcxa} and the isoperimetric constraint
\eqref{IsoConst_1}. Then, there exists $(\lambda_{0}, \lambda)\neq(0,0)$ such
that, if we define the function $F:[a,b]\times \mathbb{R}^2 \to\mathbb{R}$ by
$F=\lambda_{0} L-\lambda g$, $(x,T)$ satisfies the following fractional
Euler--Lagrange equations:
\begin{equation*}
\partial_2 F[x]_\gamma^{\alpha, \beta}(t)
+D{_{\overline{\gamma},T}^{\b,\a}}\partial_3 F[x]_\gamma^{\alpha, \beta}(t)=0
\end{equation*}
on the interval $[a,T]$, and
\begin{equation*}
\gamma_2\left({\LDb}\partial_3F[x]_\gamma^{\alpha, \beta}(t)
-{ _TD{_t^{\b}}\partial_3F[x]_\gamma^{\alpha, \beta}(t)}\right)=0
\end{equation*}
on the interval $[T,b]$.
\end{theorem}

\begin{proof}
If $(x,T)$ does not verifies \eqref{eq:eq1} or \eqref{eq:eq2}, then the
hypothesis of Theorem~\ref{teo1} is satisfied and we prove Theorem~\ref{teo3}
considering $\lambda_0=1$. If $(x,T)$ verifies \eqref{eq:eq1} and \eqref{eq:eq2},
then we prove the result by considering $\lambda=1$ and $\lambda_{0}=0$.
\end{proof}


\subsection{Problem II}

We now consider a new isoperimetric type problem
with the isoperimetric constraint of form
\begin{equation}
\label{IsoConst_2}
\int_a^b g[x]_\gamma^{\alpha, \beta}(t)dt=C,
\end{equation}
where $C$ is a given real number.

\begin{theorem}
\label{teo4}
Suppose that $(x,T)$ gives a local minimum for functional \eqref{funct1} on $D$
subject to the  boundary condition \eqref{bcxa} and the isoperimetric constraint
\eqref{IsoConst_2}. If $(x,T)$ does not satisfies the Euler--Lagrange equation
with respect to the isoperimetric constraint, that is, the condition
\begin{equation*}
\partial_2 g[x]_\gamma^{\alpha, \beta}(t) + D_{\overline{\gamma},b}^{\b,\a}
\partial_3g[x]_\gamma^{\alpha, \beta}(t)=0, \quad t\in[a,b],
\end{equation*}
is not satisfied, then there exists $\lambda\neq0$ such that,
if we define the function $F:[a,b]\times \mathbb{R}^2 \to\mathbb{R}$ by $F=L-\lambda g$,
$(x,T)$ satisfies the fractional Euler--Lagrange equations
\begin{equation}
\label{ELeq_5}
\partial_2 F[x]_\gamma^{\alpha, \beta}(t)
+ D_{\overline{\gamma},T}^{\b,\a}\partial_3 L[x]_\gamma^{\alpha, \beta}(t)
-\lambda D_{\overline{\gamma},b}^{\b,\a}\partial_3 g[x]_\gamma^{\alpha, \beta}(t)
\end{equation}
on the interval $[a,T]$, and
\begin{equation}
\label{ELeq_6}
\gamma_2 \left(_aD_t^{\b}\partial_3 F[x]_\gamma^{\alpha, \beta}(t)
-{_TD_t^{\b}\partial_3 L[x]_\gamma^{\alpha, \beta}(t)} \right)
-\lambda \left(\partial_2 g[x]_\gamma^{\alpha, \beta}(t)
+\gamma_1 {_tD_{b}^{\a}}\partial_3 g[x]_\gamma^{\alpha, \beta}(t)\right)
\end{equation}
on the interval $[T,b]$. Moreover, $(x,T)$ satisfies
the transversality conditions
\begin{equation}
\label{CT3}
\begin{cases}
L[x]_\gamma^{\alpha, \beta}(T)+\partial_1\phi(T,x(T))+\partial_2\phi(T,x(T))x'(T) =0,\\
\left[\gamma_1 \, {_tI_T^{1-\a}\partial_3L[x]_\gamma^{\alpha, \beta}(t)}
-{\gamma_2 \, {_TI_t^{1-\b}\partial_3L[x]_\gamma^{\alpha, \beta}(t)}}
+\partial_2 \phi(t,x(t)) \right]_{t=T}=0\\
\left[-\lambda \gamma_1 {_tI_b^{1-\a}\partial_3 g[x]_\gamma^{\alpha, \beta}(t)}
+\gamma_2 \left({_TI_t^{1-\b}\partial_3 L[x]_\gamma^{\alpha, \beta}(t)}
-{_aI_t^{1-\b}\partial_3F[x]_\gamma^{\alpha, \beta}(t)} \right) \right]_{t=b}=0.
\end{cases}
\end{equation}
\end{theorem}

\begin{proof}
Similarly as done to prove Theorem~\ref{teo1}, let
$$
(x^{*},T^{*})=\left(x+\e_{1}{h_{1}}+\e_{2}{h_{2}},T+\e_{1}\Delta{T}\right)
$$
be a variation of the solution, and define
$$
i(\e_{1},\e_{2})=\int_a^b g[x^{*}]_\gamma^{\alpha, \beta}(t)dt - C.
$$
The derivative $\dfrac{\partial i}{\partial \e_{2}}$, when $\e_{1}=\e_{2}=0$, is
\begin{equation*}
\left. \dfrac{\partial i}{\partial \e_{2}} \right|_{(0,0)}=\int_a^b \left(
\partial_2 g[x]_\gamma^{\alpha, \beta}(t) h_{2}(t) \
+ \partial_3 g[x]_\gamma^{\alpha, \beta}(t) \DC h_{2}(t) \right)dt.
\end{equation*}
Integrating by parts and choosing variations such that $h_2(b)=0$, we have
$$
\left.\dfrac{\partial i}{\partial \e_{2}} \right|{(0,0)}
=\int_a^b h_{2}(t) \left[ \partial_2 g[x]_\gamma^{\alpha, \beta}(t)
+ D_{\overline{\gamma},b}^{\b,\a}\partial_3g[x]_\gamma^{\alpha, \beta}(t)
\right] dt.
$$
Thus, there exists a function $h_{2}$ such that
$$
\left.\dfrac{\partial i}{\partial \e_{2}} \right|{(0,0)}\not=0.
$$
We may apply the implicit function theorem to conclude that there exists
a subfamily of variations satisfying the integral constraint.
Consider the new function $j(\e_1,\e_2)=\mathcal{J}(x^{*},T^{*})$.
Since $j$ has a local minimum at $(0,0)$ when subject to the constraint
$i(\cdot,\cdot)=0$ and $\nabla i(0,0)\not=0$, there exists a number $\lambda$
such that
\begin{equation}
\label{eqMLgrg}
\dfrac{\partial }{\partial \e_{1}} \left(j-\lambda i\right)(0,0)=0.
\end{equation}
Let $F=L-\lambda g$. Relation \eqref{eqMLgrg} can be written as
\begin{equation*}
\begin{split}
0= &\int_a^T h_{1}(t)\left[\partial_2 F[x]_\gamma^{\alpha, \beta}(t)
+ D_{\overline{\gamma},T}^{\b,\a}\partial_3 L[x]_\gamma^{\alpha, \beta}(t)
-\lambda D_{\overline{\gamma},b}^{\b,\a}\partial_3
g[x]_\gamma^{\alpha, \beta}(t) \right]dt\\
&+ \int_T^b h_{1}(t) \left[\gamma_2 \left(_aD_t^{\b}\partial_3
F[x]_\gamma^{\alpha, \beta}(t)-{_TD_t^{\b}\partial_3
L[x]_\gamma^{\alpha, \beta}(t)} \right) \right.\\
&\qquad \qquad-\left. \lambda \left(\partial_2 g[x]_\gamma^{\alpha, \beta}(t)
+\gamma_1{_tD_b^{\a}} \partial_3 g[x]_\gamma^{\alpha, \beta}(t)\right)\right]dt\\
&+ h_{1}(T)\left[\gamma_1 \, {_tI_T^{1-\a}\partial_3L[x]_\gamma^{\alpha, \beta}(t)}
-{\gamma_2 \, {_TI_t^{1-\b}\partial_3 L[x]_\gamma^{\alpha, \beta}(t)}}
+\partial_2 \phi(t,x(t)) \right]_{t=T}\\
&+\Delta T \left[ L[x]_\gamma^{\alpha, \beta}(t)+\partial_1\phi(t,x(t))
+\partial_2\phi(t,x(t))x'(t)  \right]_{t=T}\\
&+ h_{1}(b) \left[-\lambda \gamma_1 {_tI_b^{1-\a}\partial_3
g[x]_\gamma^{\alpha, \beta}(t)}+\gamma_2 \left({_TI_t^{1-\b}
\partial_3 L[x]_\gamma^{\alpha, \beta}(t)}-{_aI_t^{1-\b}\partial_3
F[x]_\gamma^{\alpha, \beta}(t)} \right) \right]_{t=b}.
\end{split}
\end{equation*}
Considering appropriate choices of variations, we obtain the first \eqref{ELeq_5}
and the second \eqref{ELeq_6} necessary optimality conditions,
and also the transversality conditions \eqref{CT3}.
\end{proof}

Similarly to Theorem~\ref{teo3}, the following result holds.

\begin{theorem}
\label{teo5}
Suppose that $(x,T)$ gives a local minimum  for functional \eqref{funct1} on
$D$ subject to the  boundary condition \eqref{bcxa} and the isoperimetric
constraint \eqref{IsoConst_2}. Then there exists $(\lambda_{0}, \lambda)\neq(0,0)$
such that, if we define the function $F:[a,b]\times \mathbb{R}^2 \to\mathbb{R}$
by $F=\lambda_0L-\lambda g$, $(x,T)$ satisfies the fractional Euler--Lagrange equations
\begin{equation*}
\partial_2 F[x]_\gamma^{\alpha, \beta}(t)
+ D_{\overline{\gamma},T}^{\b,\a}\partial_3 L[x]_\gamma^{\alpha, \beta}(t)
-\lambda D_{\overline{\gamma},b}^{\b,\a}\partial_3 g[x]_\gamma^{\alpha, \beta}(t)=0
\end{equation*}
on the interval $[a,T]$, and
\begin{equation*}
\gamma_2 \left(_aD_t^{\b}\partial_3 F[x]_\gamma^{\alpha, \beta}(t)
-{_TD_t^{\b}\partial_3 L[x]_\gamma^{\alpha, \beta}(t)} \right)
-\lambda \left(\partial_2 g[x]_\gamma^{\alpha, \beta}(t)
+\gamma_1 {_tD_{b}^{\a}}\partial_3 g[x]_\gamma^{\alpha, \beta}(t)\right)=0
\end{equation*}
on the interval $[T,b]$.
\end{theorem}


\subsection{An example}

Let $\alpha(t,\t)=\alpha(t)$ and $\beta(t,\t)=\beta(\t)$. Define the function
$$
\psi(T)=\int_0^T \left(\dfrac{t^{1-\alpha (t)}}{2 \Gamma(2-\alpha(t))}
+\dfrac{(b-t)^{1-\beta (t)}}{2 \Gamma(2-\beta(t))} \right)^2dt
$$
on the interval $[0,b]$ with $b>0$. Consider the functional $J$ defined by
$$
J(x,t)= \int_0^T \left[\alpha(t)+\left(\DC x(t) \right)^2
+ \left(\dfrac{t^{1-\alpha (t)}}{2 \Gamma(2-\alpha(t))}
+\dfrac{(b-t)^{1-\beta (t)}}{2 \Gamma(2-\beta(t))}\right)^2 \right]dt
$$
for $t \in [0,b]$ and $\gamma =(1/2,1/2)$, subject to the initial condition
$$
x(0)=0
$$
and the isoperimetric constraint
$$
\int_0^T \DC x(t)  \left(\dfrac{t^{1-\alpha (t)}}{2 \Gamma(2-\alpha(t))}
+\dfrac{(b-t)^{1-\beta (t)}}{2 \Gamma(2-\beta(t))}\right)^2 dt=\psi(T).
$$
Define $F=L-\lambda g$ with $\lambda=2$, that is,
$$
F=\alpha(t)+\left({^CD_\gamma^{\alpha(\cdot),\beta(\cdot)}} x(t)
-\frac{t^{1-\alpha(t)}}{2\Gamma(2-\alpha(t))}
-\frac{(b-t)^{1-\beta(t)}}{2\Gamma(2-\beta(t))}\right)^2.
$$
Consider the function $\overline{x}(t)=t$ with $t\in[0,b]$. Because
$$
\DC \overline{x}(t)=\frac{t^{1-\alpha(t)}}{2\Gamma(2-\alpha(t))}
+\frac{(b-t)^{1-\beta(t)}}{2\Gamma(2-\beta(t))},
$$
we have that $\overline x$ satisfies conditions \eqref{ELeq_1}, \eqref{ELeq_2}
and the two last of \eqref{CT1}. Using the first condition of \eqref{CT1}, that is,
$$
\alpha(t)+2\left(\dfrac{T^{1-\alpha (T)}}{2 \Gamma(2-\alpha(T))}
+\dfrac{(b-T)^{1-\beta (T)}}{2 \Gamma(2-\beta(T))} \right)^2=0,
$$
we obtain the optimal time $T$.


\section{Holonomic constraints}
\label{sec:Hol}

Consider the space
\begin{equation}
\label{space_HL1}
U=\lbrace(x_1,x_2,T) \in C^1([a,b]) \times C^1([a,b]) \times[a,b]
: x_1(a)=x_{1a} \wedge x_2(a)=x_{2a}\rbrace
\end{equation}
for fixed reals $x_{1a},x_{2a} \in \mathbb{R}$.
In this section we consider the functional $\mathcal{J}$ defined in $U$ by
\begin{equation}
\label{functHL}
\mathcal{J}(x_1,x_2,T)=\int_a^T L(t,x_1(t),x_2(t), \DC x_1(t), \DC x_2(t))dt
+ \phi(T,x_1(T),x_2(T))
\end{equation}
with terminal time $T$ and terminal states $x_1(T)$ and $x_2(T)$ free.
The Lagrangian $L:[a,b]\times \mathbb{R}^4 \rightarrow \mathbb{R}$ is a
continuous function and continuously differentiable with respect to the its
$i$th argument, $i \in \lbrace 2,3,4,5\rbrace$. To define the variational problem,
we consider a new constraint of the form
\begin{equation}
\label{HLConst}
g(t,x_1(t),x_2(t))=0, \quad t\in [a,b],
\end{equation}
where $g: [a,b]\times\mathbb{R}^2 \rightarrow \mathbb{R}$ is a continuous
function and continuously differentiable with respect to second and third arguments.
This constraint is called a holonomic constraint. The next theorem gives
fractional necessary optimality conditions to the variational problem with
a holonomic constraint. To simplify the notation, we denote by $x$ the vector
$(x_1,x_2)$; by $\DC x$ the vector $(\DC x_1,\DC x_2)$; and we use the operator
$$
[x]_\gamma^{\alpha, \beta}(t):=\left(t, x(t), \DC x(t)\right).
$$

\begin{theorem}
\label{teo7}
Suppose that $(x,T)$ gives a local minimum to functional $\mathcal{J}$ as in
\eqref{functHL}, under the constraint \eqref{HLConst} and the
boundary conditions defined in \eqref{space_HL1}. If
$$
\partial_3g(t,x(t))\neq0 \quad \forall t\in [a,b],
$$
then there exists a piecewise continuous function
$\lambda : [a,b]\rightarrow \mathbb{R}$ such that $(x,T)$
satisfies the following fractional Euler--Lagrange equations:
\begin{equation}
\label{ELeqh_1}
\partial_2 L[x]_\gamma^{\alpha, \beta}(t) + D_{\overline{\gamma},T}^{\b,\a}
\partial_4L[x]_\gamma^{\alpha, \beta}(t) + \lambda(t) \partial_2 g(t,x(t))=0
\end{equation}
and
\begin{equation}
\label{ELeqhPP_1}
\partial_3 L[x]_\gamma^{\alpha, \beta}(t) + D_{\overline{\gamma},T}^{\b,\a}
\partial_5L[x]_\gamma^{\alpha, \beta}(t) + \lambda(t)\partial_3 g(t,x(t))=0
\end{equation}
on the interval $[a,T]$, and
\begin{equation}
\label{ELeqh_2}
\gamma_2\left(_aD_t^{\b} \partial_4 L[x]_\gamma^{\alpha, \beta}(t)
- {_TD_t}^{\b} \partial_4 L[x]_\gamma^{\alpha, \beta}(t)
+ \lambda(t)\partial_2 g(t,x(t))\right)=0
\end{equation}
and
\begin{equation}
\label{ELeqhPP_2}
_aD_t^{\b} \partial_5 L[x]_\gamma^{\alpha, \beta}(t)
- {_TD_t}^{\b} \partial_5 L[x]_\gamma^{\alpha, \beta}(t)
+\lambda(t)\partial_3 g(t,x(t))=0
\end{equation}
on the interval $[T,b]$. Moreover, $(x,T)$ satisfies the transversality conditions
\begin{equation}
\label{CTh1}
\begin{cases}
L[x]_\gamma^{\alpha, \beta}(T)  + \partial_1 \phi(T,x(T))
+ \partial_2 \phi(T,x(T))x'_1(T)+ \partial_3 \phi(T,x(T))x'_2(T) =0,\\
\left[\gamma_1 \, {_tI_T^{1-\a}\partial_4L[x]_\gamma^{\alpha, \beta}(t)}
-{\gamma_2 \, {_TI_t^{1-\b}\partial_4L[x]_\gamma^{\alpha, \beta}(t)}}\right]_{t=T}
+\partial_2 \phi(T,x(T))=0,\\
\left[\gamma_1 {{_tI_T}}^{1-\a} \partial_5 L[x]_\gamma^{\alpha, \beta}(t)
- \gamma_2 {_TI_t}^{1-\b} \partial_5 L[x]_\gamma^{\alpha, \beta}(t)\right]_{t=T}
+\partial_3 \phi(T,x(T))=0,\\
\gamma_2 \left[ {_TI_t^{1-\b}}\partial_4 L[x]_\gamma^{\alpha, \beta}(t)
-{_aI_t^{1-\b}\partial_4 L[x]_\gamma^{\alpha, \beta}(t)}\right]_{t=b}=0\\
\gamma_2 \left[ {_TI_t^{1-\b}}\partial_5 L[x]_\gamma^{\alpha, \beta}(t)
-{_aI_t^{1-\b}\partial_5 L[x]_\gamma^{\alpha, \beta}(t)}\right]_{t=b}=0.
\end{cases}
\end{equation}
\end{theorem}

\begin{proof}
Consider admissible variations of the optimal solution $(x,T)$ of the type
\begin{equation*}
(x^{*},T^{*})=\left(x+\e h,T+\e \Delta{T}\right),
\end{equation*}
where $\e \in\mathbb{R}$ is a small parameter, $h=(h_1,h_2)\in C^1([a,b])
\times C^1([a,b])$ satisfies $h_{i}(a)=0$, $i=1,2$, and $\triangle T \in\mathbb{R}$.
Because
$$
\partial_3g(t,x(t))\neq0 \quad \forall t\in [a,b],
$$
by the implicit function theorem there exists a subfamily of variations of
$(x,T)$ that satisfy \eqref{HLConst}, that is, there exists a unique function
$h_2(\e ,h_1)$ such that the admissible variation ($x^{*}, T^*$) satisfies
the holonomic constraint \eqref{HLConst}:
$$
g(t,x_1(t)+\e h_1(t),x_2(t)+\e h_2)=0 \quad \forall t\in [a,b].
$$
Differentiating this condition with respect to $\e$
and considering $\e =0$, we obtain that
\begin{equation*}
\partial_2 g(t,x(t)) h_{1}(t) + \partial_3 g(t,x(t)) h_{2}(t) = 0,
\end{equation*}
which is equivalent to
\begin{equation}
\label{eqh1}
\dfrac{\partial_2 g(t,x(t)) h_{1}(t)}{\partial_3 g(t,x(t))}=-h_{2}(t).
\end{equation}
Define $j$ on a neighbourhood of zero by
$$
j(\e )=\int_a^{T+\e  \triangle T} L[x^{*}]_\gamma^{\alpha, \beta}(t)dt
+\phi(T+\e \triangle T, x^{*}(T+\e \triangle T)).
$$
The derivative $\dfrac{\partial j}{\partial \e }$ for $\e =0$ is
\begin{equation}
\label{eqh2}
\begin{split}
\left. \dfrac{\partial j}{\partial \e } \right|_{\e=0}
= \int_a^{T} & \left(\partial_2 L[x]_\gamma^{\alpha, \beta}(t) h_{1}(t)
+\partial_3 L[x]_\gamma^{\alpha, \beta}(t) h_{2}(t) \right. \\
& \left.+ \partial_4 L[x]_\gamma^{\alpha, \beta}(t) \DC h_{1}(t)
+ \partial_5 L[x]_\gamma^{\alpha, \beta}(t) \DC h_{2}(t)\right)dt\\
& +L[x]_\gamma^{\alpha, \beta}(T) \triangle T + \partial_1 \phi(T,x(T))\triangle T
+ \partial_2 \phi(T,x(T))\left[h_1(T)+x'_1(T)\triangle T\right]\\
& + \partial_3 \phi(T,x(T))\left[h_2(T)+x'_2(T)\triangle T\right].
\end{split}
\end{equation}
The third term in \eqref{eqh2} can be written as
\begin{equation}
\label{term3}
\begin{split}
\int_a^T  &\partial_4 L[x]_\gamma^{\alpha, \beta}(t) \DC h_{1}(t) dt\\
&=\int_a^T \partial_4 L[x]_\gamma^{\alpha, \beta}(t)\left[\gamma_1
\, \LC h_{1}(t)+\gamma_2 \, \RCb h_{1}(t)\right]dt\\
&=\gamma_1 \int_a^T \partial_4 L[x]_\gamma^{\alpha, \beta}(t)\LC h_{1}(t)dt\\
&\quad + \gamma_2 \left[ \int_a^b
\partial_4 L[x]_\gamma^{\alpha, \beta}(t)\RCb h_{1}(t)dt
- \int_T^b \partial_4 L[x]_\gamma^{\alpha, \beta}(t) \RCb h_{1}(t)dt \right].
\end{split}
\end{equation}
Integrating by parts, \eqref{term3} can be written as
\begin{equation*}
\begin{split}
\int_a^T h_{1}(t) & \left[ \gamma_1 {_tD_T}^{\a} \partial_4
L[x]_\gamma^{\alpha, \beta}(t) + \gamma_2 {_aD_t}^{\b} \partial_4
L[x]_\gamma^{\alpha, \beta}(t) \right] dt\\
+ &  \int_T^b \gamma_2 h_{1}(t) \left[ _aD_t^{\b} \partial_4
L[x]_\gamma^{\alpha, \beta}(t) - {_TD_t}^{\b} \partial_4
L[x]_\gamma^{\alpha, \beta}(t) \right] dt\\
+ & \Biggl[ h_{1}(t) \left( \gamma_1 {{_tI_T}}^{1-\a} \partial_4
L[x]_\gamma^{\alpha, \beta}(t) - \gamma_2 {_TI_t}^{1-\b} \partial_4
L[x]_\gamma^{\alpha, \beta}(t) \Biggr) \right]_{t=T}\\
+ & \Biggl[ \gamma_2 h_{1}(t)  \left( {{_TI_t}}^{1-\b} \partial_4
L[x]_\gamma^{\alpha, \beta}(t) - {_aI_t}^{1-\b} \partial_4
L[x]_\gamma^{\alpha, \beta}(t) \Biggr) \right]_{t=b}.
\end{split}
\end{equation*}
By proceeding similarly to the $4$th term in \eqref{eqh2},
we obtain an equivalent expression. Substituting these relations into
\eqref{eqh2} and considering the fractional operator
$D_{\overline{\gamma},c}^{\b,\a}$ as defined in \eqref{aux:FD}, we obtain that
\begin{equation}
\label{eqh3}
\begin{split}
0= \int_a^T &\Biggl[ h_{1}(t) \left[ \partial_2
L[x]_\gamma^{\alpha, \beta}(t) + D_{\overline{\gamma},T}^{\b,\a}\partial_4
L[x]_\gamma^{\alpha, \beta}(t) \right]\biggr. \\
& \qquad + \biggl. h_{2}(t) \left[ \partial_3 L[x]_\gamma^{\alpha, \beta}(t)
+ D_{\overline{\gamma},T}^{\b,\a}\partial_5
L[x]_\gamma^{\alpha, \beta}(t) \right] \biggr]dt\\
+ &  \gamma_2\int_T^b  \Biggl[h_{1}(t) \left[ _aD_t^{\b} \partial_4
L[x]_\gamma^{\alpha, \beta}(t) - {_TD_t}^{\b} \partial_4
L[x]_\gamma^{\alpha, \beta}(t) \right] \biggr.\\
& \qquad +\Biggl.h_{2}(t) \left[ _aD_t^{\b} \partial_5
L[x]_\gamma^{\alpha, \beta}(t) - {_TD_t}^{\b} \partial_5
L[x]_\gamma^{\alpha, \beta}(t) \right]dt \biggr]\\
+ & h_{1}(T)\Biggl[  \gamma_1 {{_tI_T}}^{1-\a} \partial_4
L[x]_\gamma^{\alpha, \beta}(t) - \gamma_2 {_TI_t}^{1-\b} \partial_4
L[x]_\gamma^{\alpha, \beta}(t) +\partial_2 \phi(t,x(t))\Biggr] _{t=T}\\
+ & h_{2}(T)\Biggl[   \gamma_1 {{_tI_T}}^{1-\a} \partial_5
L[x]_\gamma^{\alpha, \beta}(t) - \gamma_2 {_TI_t}^{1-\b} \partial_5
L[x]_\gamma^{\alpha, \beta}(t) + \partial_3 \phi(t,x(t))\Biggr]_{t=T}\\
+ &\triangle T \bigg[ L[x]_\gamma^{\alpha, \beta}(t) + \partial_1 \phi(t,x(t))
+ \partial_2 \phi(t,x(t))x'_1(t)
+ \partial_3 \phi(t,x(t))x'_2(t) \Biggr]_{t=T}\\
+ & h_{1}(b)\Biggl[ \gamma_2   \left( {{_TI_t}}^{1-\b} \partial_4
L[x]_\gamma^{\alpha, \beta}(t) - {_aI_t}^{1-\b} \partial_4
L[x]_\gamma^{\alpha, \beta}(t) \Biggr) \right]_{t=b}\\
+ & h_{2}(b)\Biggl[ \gamma_2   \left( {{_TI_t}}^{1-\b} \partial_5
L[x]_\gamma^{\alpha, \beta}(t) - {_aI_t}^{1-\b} \partial_5
L[x]_\gamma^{\alpha, \beta}(t) \Biggr) \right]_{t=b}.
\end{split}
\end{equation}
Define the piecewise continuous function $\lambda$ by
\begin{equation}
\label{lambdah1}
\lambda (t)
=
\begin{cases}
-\dfrac{\partial_3L[x]_\gamma^{\alpha, \beta}(t)
+ D_{\overline{\gamma},T}^{\b,\a}\partial_5L[x]_\gamma^{\alpha,
\beta}(t)}{\partial_3 g(t,x(t))}, & t\in [a,T] \\
-\dfrac{_aD_t^{\b}\partial_5L[x]_\gamma^{\alpha, \beta}(t)
-_TD_t^{\b}\partial_5L[x]_\gamma^{\alpha, \beta}(t)}{\partial_3 g(t,x(t))},
& t\in [T,b].
\end{cases}
\end{equation}
Using equations \eqref{eqh1} and \eqref{lambdah1}, we obtain that
\begin{equation*}
\lambda(t)\partial_2 g(t,x(t)) h_{1}(t)
=
\begin{cases}
(\partial_3L[x]_\gamma^{\alpha, \beta}(t) +D_{\overline{\gamma},T}^{\b,\a}
\partial_5L[x]_\gamma^{\alpha, \beta}(t)) h_2(t), & t\in [a,T] \\
(_aD_t^{\b}\partial_5L[x]_\gamma^{\alpha, \beta}(t)
-_TD_t^{\b}\partial_5L[x]_\gamma^{\alpha, \beta}(t)) h_2(t),  & t\in [T,b].
\end{cases}
\end{equation*}
Substituting in \eqref{eqh3}, we have
\begin{equation*}
\begin{split}
0= & \int_a^T h_{1}(t) \left[ \partial_2 L[x]_\gamma^{\alpha, \beta}(t)
+ D_{\overline{\gamma},T}^{\b,\a}\partial_4L[x]_\gamma^{\alpha, \beta}(t)
+ \lambda(t)\partial_2 g(t,x(t))\right] dt\\
+ &  \gamma_2\int_T^b  h_{1}(t) \left[ _aD_t^{\b} \partial_4
L[x]_\gamma^{\alpha, \beta}(t) - {_TD_t}^{\b} \partial_4
L[x]_\gamma^{\alpha, \beta}(t) + \lambda(t)\partial_2 g(t,x(t))\right] dt \\
+ & h_{1}(T)\Biggl[  \gamma_1 {{_tI_T}}^{1-\a} \partial_4
L[x]_\gamma^{\alpha, \beta}(t) - \gamma_2 {_TI_t}^{1-\b} \partial_4
L[x]_\gamma^{\alpha, \beta}(t) +\partial_2 \phi(t,x(t))\Biggr] _{t=T}\\
+ & h_{2}(T)\Biggl[   \gamma_1 {{_tI_T}}^{1-\a} \partial_5
L[x]_\gamma^{\alpha, \beta}(t) - \gamma_2 {_TI_t}^{1-\b} \partial_5
L[x]_\gamma^{\alpha, \beta}(t) + \partial_3 \phi(t,x(t))\Biggr]_{t=T}\\
+ &\triangle T \bigg[ L[x]_\gamma^{\alpha, \beta}(t)
+ \partial_1 \phi(t,x(t))+ \partial_2 \phi(t,x(t))x'_1(t)
+ \partial_3 \phi(t,x(t))x'_2(t) \Biggr]_{t=T}\\
\end{split}
\end{equation*}
\begin{equation*}
\begin{split}
+ & h_{1}(b)\Biggl[ \gamma_2   \left( {{_TI_t}}^{1-\b} \partial_4
L[x]_\gamma^{\alpha, \beta}(t) - {_aI_t}^{1-\b} \partial_4
L[x]_\gamma^{\alpha, \beta}(t) \Biggr) \right]_{t=b}\\
+ & h_{2}(b)\Biggl[ \gamma_2   \left( {{_TI_t}}^{1-\b} \partial_5
L[x]_\gamma^{\alpha, \beta}(t) - {_aI_t}^{1-\b} \partial_5
L[x]_\gamma^{\alpha, \beta}(t) \Biggr) \right]_{t=b}.
\end{split}
\end{equation*}
Considering appropriate choices of variations, we obtained the first \eqref{ELeqh_1}
and the third \eqref{ELeqh_2} necessary conditions, and also the transversality
conditions \eqref{CTh1}. The remaining conditions \eqref{ELeqhPP_1}
and \eqref{ELeqhPP_2} follow directly from \eqref{lambdah1}.
\end{proof}

We end this section with a simple illustrative example.
Consider the following problem:
\begin{equation*}
\begin{gathered}
J(x,t)= \int_0^T \Biggl[\alpha(t)+\left(\DC x_1(t) 
-\dfrac{t^{1-\alpha(t)}}{2 \Gamma(2-\alpha(t))}
-\dfrac{(b-t)^{1-\beta (t)}}{2 \Gamma(2-\beta(t))}\right)^2 \\
+ \left(\DC x_2(t)\right)^2\Biggr]dt \longrightarrow\min,\\
x_1(t) + x_2(t) = t+1,\\
x_1(0)=0, \quad x_2(0) = 1.
\end{gathered}
\end{equation*}
It is a simple exercise to check that $x_1(t) = t$,
$x_2(t) \equiv 1$ and $\lambda(t) \equiv 0$ 
satisfy our Theorem~\ref{teo7}.


\section{Conclusion}

Nowadays, optimization problems involving fractional derivatives constitute 
a very active research field due to several applications \cite{MR3443073,MR3331286,MR2984893}.
Here we obtained optimality conditions for two isoperimetric problems
and for a new variational problem subject to a holonomic constraint, 
where the Lagrangian depends on a combined Caputo derivative of 
variable fractional order. Main results include Euler--Lagrange  
and transversality type conditions. For simplicity, we considered 
here only linear combinations between the left and right operators. 
Using similar techniques as the ones developed here, one can obtain 
analogous results for fractional variational problems with Lagrangians
containing left-sided and right-sided fractional derivatives of variable order.
More difficult and interesting, would be to develop
a ``multi-term fractional calculus of variations''.
The question seems however nontrivial, even for the nonvariable order case, 
because of difficulties in application of integration by parts.
For the variable order case, as we consider in our work, 
there is yet no formula of fractional integration by parts
for higher-order derivatives. This is under investigation
and will be addressed elsewhere.


\section*{Acknowledgments}

This work is part of first author's Ph.D., which is carried out at the
University of Aveiro under the Doctoral Programme \emph{Mathematics
and Applications} of Universities of Aveiro and Minho.
It was supported by Portuguese funds through the
\emph{Center for Research and Development in Mathematics and Applications} (CIDMA),
and the \emph{Portuguese Foundation for Science and Technology} (FCT),
within project UID/MAT/04106/2013. Tavares was also supported
by FCT through the Ph.D. fellowship SFRH/BD/42557/2007.
The authors are grateful to four Reviewers, for several
pertinent questions and remarks, which improved the final version
of the manuscript.



\end{document}